\documentclass[12pt]{amsart}
\usepackage{amssymb}
\usepackage{mathtools}
\usepackage[english]{babel}
\usepackage{xcolor}

\usepackage{epsfig}
\usepackage{mathptmx}
\usepackage{amsmath,amsfonts,amssymb}
\usepackage{mathtools}
\usepackage{mathrsfs}
\usepackage[all]{xy}
\usepackage{graphicx}
\usepackage{latexsym}
\usepackage{verbatim}

\numberwithin{equation}{section}

\def\Re{{\sf Re}\,}

\newcommand{\D}{\mathbb D}

\newcommand{\R}{\mathbb R}

\newcommand{\C}{\mathbb C}

\newcommand{\B}{\mathbb B}

\newcommand{\N}{\mathbb N}

\def\Re{{\sf Re}\,}

\def\Re{{\sf Re}\,}

\def\re{\operatorname{Re}}

\def\Re{{\sf Re}\,}

\def\1#1{\overline{#1}}
\def\2#1{\widetilde{#1}}
\def\3#1{\widehat{#1}}
\def\4#1{\mathbb{#1}}
\def\5#1{\frak{#1}}
\def\6#1{{\mathcal{#1}}}

\def\Re{{\sf Re}\,}

\newcommand{\mcite}[1]{\csname b@#1\endcsname}

\theoremstyle{theorem}

\def\Re{{\sf Re}\,}

\newtheorem{theorem}{Theorem}[section]
\newtheorem{lemma}[theorem]{Lemma}
\newtheorem{proposition}[theorem]{Proposition}

\theoremstyle{definition}

\newtheorem{example}[theorem]{Example}

\theoremstyle{remark}
\newtheorem{remark}[theorem]{Remark}

\numberwithin{equation}{section}

\title[geodesic rigidity]{Holomorphic maps acting as Kobayashi isometries on a family of geodesics}

\author[F. Bracci]{Filippo Bracci$^\dag$}
\address{F. Bracci: Dipartimento di Matematica, Universit\`a di Roma ``Tor Vergata", Via della Ricerca
	Scientifica 1, 00133, Roma, Italia.} \email{fbracci@mat.uniroma2.it}
	
\author[\L. Kosi\'nski]{\L ukasz Kosi\'nski$^{\dag\dag}$}

\address{\L. Kosi\'nski: Institute of Mathematics, Faculty of Mathematics and Computer Science, Jagiellonian University, \L ojasiewicza 6, 30-348 Krak\'ow, Poland} \email{lukasz.kosinski@uj.edu.pl}

\author[W. Zwonek]{W\l odzimierz Zwonek$^{\dag\dag\dag}$}
\address{W. Zwonek: Institute of Mathematics, Faculty of Mathematics and Computer Science, Jagiellonian University, \L ojasiewicza 6, 30-348 Krak\'ow, Poland} \email{wlodzimierz.zwonek@uj.edu.pl}


\thanks{$^\dag$Partially supported by  PRIN {\sl  Real and Complex Manifolds: Geometry and Holomorphic Dynamics} n. 2022AP8HZ9, by INdAM, and by the   MUR Excellence Department Project MatMod@TOV
CUP:E83C23000330006}
\thanks{$^{\dag\dag}$ Partially supported by Sheng grant no. 2023/48/Q/ST1/00048 of the National Science Center, Poland}
\thanks{$^{\dag\dag\dag}$ Partially supported by the Preludium bis grant no. 2021/43/O/ST1/02111 of the National Science Centre, Poland}

\long\def\REM#1{\relax}

\begin{document}
	\maketitle

	\selectlanguage{english}
	\begin{abstract}
		Consider a holomorphic map $F: D \to G$ between two domains in ${\mathbb C}^N$. Let $\mathcal F$ denote a family of geodesics for the Kobayashi distance, such that $F$ acts as an isometry on each element of $\mathcal F$. This paper is dedicated to characterizing the scenarios in which the aforementioned condition implies that $F$ is a biholomorphism. Specifically, we establish this when $D$ is a complete hyperbolic domain, and $\mathcal F$ comprises all geodesic segments originating from a specific point. Another case is when $D$ and $G$ are $C^{2+\alpha}$-smooth bounded pseudoconvex domains, and $\mathcal F$ consists of all geodesic rays converging at a designated boundary point of $D$. Furthermore, we provide examples to demonstrate that these assumptions are essentially optimal.	\end{abstract}
	
	\section{Introduction}
Understanding the conditions under which a holomorphic map between two complex manifolds becomes a biholomorphism has been a longstanding and intriguing question in the literature. The vastness of the literature makes it challenging to provide a comprehensive overview. Notably, classical results such as the Schwarz lemma, along with its generalizations by Carath\'eodory, Cartan, Kaup, and Wu, offer examples of simple conditions that distinguish biholomorphisms among holomorphic maps (see, e.g., \cite{A2}). These theorems also have natural counterparts in the context of invariant metrics and (pluri)complex Green and Poisson functions (\cite{A2, BCD, BP, BDK, Bra-Kra-Rot 2020}).	

When dealing with bounded domains in $\mathbb{C}^N$, such results extend to intriguing natural boundary counterparts. Julia's lemma, rooted in versions of the maximum principle for (sub)harmonic functions, and its various generalizations (see, e.g., \cite{BCD, A2}) offer geometric conditions. To some extent, these conditions have been generalized in different ways to higher dimensions (see \cite{Ru}).

In our previous work \cite{BKZ}, we delved into the question of when, given two convex domains and a family of complex geodesics, a holomorphic map that is isometric with respect to the Kobayashi distance for any complex geodesic in the given family is indeed an automorphism. This question is meaningful specifically within the class of convex domains, as it remains unclear outside this class whether a domain possesses a sufficient number of complex geodesics.

Nevertheless, in a complete Kobayashi hyperbolic domain, every two points can be connected by a geodesic ({\sl i.e.}, a length-minimizing curve), and biholomorphisms naturally preserve these geodesics, in the sense that, if $\gamma:I \to M$ is a  geodesic in $M$, and $F:M\to N$ is a biholomorphism, then $F\circ \gamma:I\to N$ is a  geodesic in $N$. In particular, 
	\begin{equation}\label{eq:pres-geo-def}
	K_M(\gamma(t),\gamma(s))=K_N(F(\gamma(t)),F(\gamma(s))), \quad \hbox{ for all } t,s\in I,
	\end{equation}
where, for a complex manifold $X$,  $K_X$ denotes its Kobayashi distance.

Hence, it is pertinent to inquire about the extent to which \eqref{eq:pres-geo-def} characterizes biholomorphisms for a given family of geodesics.

The response proves to be remarkably broad, particularly when one examines the family comprising all geodesic segments emanating from a specific point:

\begin{theorem}\label{main-intro-inner}
		Let $D\subset \C^N$, $N\geq 1$, be a complete hyperbolic domain  and let $G$ be a domain in $\C^N$. Let  $p\in D$ and let $\mathcal F$ be the family of all geodesic segments in $D$  starting from $p$. Let $F:D\to G$ be holomorphic. Then $F$ is a biholomorphism if and only if  $F\circ \gamma$ is a  geodesic of $G$ for all $\gamma\in\mathcal F$.
	\end{theorem}
	
When addressing the same question in the context of the boundary, caution must be exercised in selecting the framework, as geodesics may exhibit erratic behavior near the boundary. In situations where geodesic rays consistently "land" at a specific point (as observed in $C^2$ strongly pseudoconvex domains), it is meaningful to investigate the persistence of the previous result. Indeed, we can establish:

	\begin{theorem}\label{main-intro}
		Let $D, G\subset \C^N$, $N\geq 1$, be $C^{2+\alpha}$-smooth bounded strongly pseudoconvex domains, $\alpha>0$, and let $p\in \partial{D}$. Let $\mathcal F$ be the family of all geodesic rays $\gamma:[0,+\infty)\to D$ such that $\lim_{t\to+\infty}\gamma(t)=p$. Let $F:D\to G$ be holomorphic. Then $F$ is a biholomorphism if and only if  $F\circ \gamma$ is a  geodesic of $G$ for all $\gamma\in\mathcal F$.
	\end{theorem}
	
The key characteristic of the family $\mathcal F$ of geodesics examined in the preceding theorems is that, for any $z\in D$, there exists a geodesic in the family containing it in its image. A family of geodesics possessing this property is termed {\sl complete}.

Hence, a natural question arises: Do the earlier results apply universally to every complete family of geodesics? The response is in the negative, and our assumptions are found to be fundamentally optimal. We will demonstrate examples of such scenarios in Section~\ref{section:special-cases}.

	\section{Geodesics in strongly pseudoconvex domains}\label{Sec:geo-pseudo}

Let $M$ be a complex manifold and let $K_M$ denote its {\sl Kobayashi pseudodistance} and let $k_M$ be its {\sl infinitesimal Kobayashi pseudometric}. Let $I\subset \R$ be an interval. An absolutely continuous curve $\gamma:I\to M$ is a {\sl geodesic} for $k_M$ if for every $s, t\in I$, $s\leq t$, 
	\[
	K_M(\gamma(s),\gamma(t))=\ell_M(\gamma;[s,t]):=\int_s^t k_M(\gamma(u);\gamma'(u))du.
	\]
	Any geodesic $\gamma$ can be reparametrized in (hyperbolic) arc-length, so that $K_M(\gamma(s), \gamma(t))=|t-s|$. If $\gamma$ is parametrized by arc-length and its interval of definition $I$ is a closed interval, then $\gamma$ is called a {\sl geodesic segment}. If $I=\R^+$, then $\gamma$ is a {\sl geodesic ray}, while, if $I=\R$, then $\gamma$ is a {\sl geodesic line}. We just say that $\gamma$ is a geodesic when there is no need to specify its domain of definition. 	
	If $M$ is complete hyperbolic, by Hopf-Rinow's Theorem, given any two points $p, q\in M$ there always exists a geodesic segment $\gamma: [a,b]\to M$ which joins $p$ and $q$, namely, such that $\gamma(a)=p$ and $\gamma(b)=q$.

 If $D\subset \C^N$, we say that a geodesic ray $\gamma:[0,+\infty)\to D$  lands at some point $p\in\partial D$ provided $\lim_{t\to+\infty}\gamma(t)=p$.

We collect below the properties we need for geodesics in strongly pseudoconvex domains. These are well known, and we only sketch the proof.

	\begin{proposition}\label{Prop:geo-in-sp}
		Let $D\subset \C^N$, $N\geq 2$, be a bounded strongly pseudoconvex domain with $C^{2}$ boundary.  Then:
		\begin{enumerate}
			\item  for every two points $p,q\in \overline{D}$  exists a geodesic segment, ray or line $\gamma: I \to D$ such that $p,q\in \overline{\gamma(I)}$.			\item Let $p\in \overline{D}$ and $\{z_n\}\subset D$ be a sequence converging to $q\in\partial D\setminus\{p\}$. If $\{\gamma_n\}$ is a sequence of geodesic segments/rays joining $p$ with $z_n$, then, up to affine reparametrizations, $\{\gamma_n\}$ converges uniformly  to a geodesic ray/line joining $p$ with $q$.
			\item if $\gamma:[0,+\infty)\to D$ is a geodesic ray, then there exists $p\in \partial D$ such that  $\lim_{t\to+\infty}\gamma(t)=p$. 
			\item If $\gamma, \eta$ are two geodesic rays in $D$ converging to $p\in \partial D$ then there exists $C>0$ such that $\lim_{t\to+\infty}K_D(\gamma(t), \eta(t))\leq C$.
		\end{enumerate}
	\end{proposition}

\begin{proof}
By \cite{BaBo}, $D$ is Gromov hyperbolic with respect to the Kobayashi distance and the identity map extends as a homeomorphism from the Gromov closure of $D$ and $\overline{D}$, that is, $D$ is a so-called {\sl Gromov model domain} (see \cite{BG}). In particular it is ``visible'' hence (1), (2) and (3) follows (for the uniform convergence in (2) see  \cite[Lemma~3.1]{Bra}).

In order to prove (4),  let $\gamma, \gamma_0$ be two geodesic rays converging to $p\in\partial D$. Since $D$ is a Gromov model domain, it has no ``geodesic loops'' and hence by \cite[Prop. 4.4]{Bra} there exists $C>0$ such that
\[
K_D(\gamma(t), \gamma_0(t))\leq C, \quad \forall t\geq 0,
\]
and we are done.
\end{proof}
	
	\subsection{Application of the scaling method}\label{Sec:scaling}
	
	Let $\Omega$ be a bounded strongly pseudoconvex domain with $C^{2+\alpha}$-smooth boundary, $\alpha\in (0,1]$, and let $p\in \partial \Omega$. 
	
	It follows from Lempert~\cite{Lem 1981} and Fridman-Ma \cite{Fridman} (see also \cite[Section~2]{Kos-Nik-Tho 2023} for details) that there exists a mapping $T$ biholomorphic on $\overline \Omega$ sending $p$ to $e_1=(1,0, \ldots, 0)$ such that $\Omega_0:=T(\Omega)$ is contained in the ball of radius $1+\delta_0$ and centered at $(-\delta_0,0)$, where $\delta_0\geq 0$, and $\partial \Omega_0$ has  a defining function near $e_1$ of the form
	\begin{equation}\label{eq:def}
	\rho(z) = -1 +|z|^2+ o(|z-e_1|^{2+\alpha}).
	\end{equation}
	Writing $z=(z_1,z'')\in \C\times\C^{N-1}$, define
	\begin{equation}\label{eq:aut}A_t(z) = \left(\frac{z_1+ t}{1+ tz_1}, \sqrt{1-t^2} \frac{z'}{1 + tz_1}\right),
	\end{equation} and 
	\begin{equation}\label{eq:Dt} \Omega_t:=A_t^{-1}(\Omega_0),\quad t\in (0,1).
	\end{equation}
	Then, 
	\begin{lemma}\label{Lem:scaling}
		\begin{enumerate}
			\item for every $\beta>-1$, $\Omega_t\cap \{\re z_1>\beta\}$ converges as $t\to 1$ to $\mathbb B^N \cap \{\re z_1>\beta\}$ in $\mathcal C^{2+\alpha}$-topology,
			\item $\Omega_t$ converges to $\B^N$ in the Hausdorff topology as $t\to 1$.
		\end{enumerate}
	\end{lemma}
	
	In the sequel, we will need the following results.
	
	\begin{lemma}\label{Lem:metric-conv}
		$K_{\Omega_{t_k}}(\cdot,\cdot)$ converges uniformly on compacta to $K_{\B^N}(\cdot,\cdot)$. 
	\end{lemma}
	\begin{proof}
Fix $r_0\in (0,1)$. For $r>0$ let $\B(0,r)$ be the ball of center $0$ and radius $r$. Since, clearly, $K_{\B(0,r)}(\cdot,\cdot)$ converges uniformly on $\B(0,r_0)$ to $K_{\B^N}(\cdot,\cdot)$ for $r\to 1^-$, for every $\delta>0$ we can find $r_1\in (r_0,1)$ such that $K_{\B(0,r_1)}(z,w)\leq K_{\B^N}(z,w)+\delta$ for all $z,w\in \B(0,r_0)$. But, for $k$ large enough, $\B(0,r_1)\subset \Omega_{t_k}$. Thus, for $k$ large enough, and for all $z,w\in \B(0,r_0)$,
		\[
		 K_{\Omega_{t_k}}(z,w)\leq K_{\B(0,r_1)}(z,w)\leq K_{\B^N}(z,w)+\delta.
		\]
Similarly, we show that \[
		k_{\B^N}(z,w) -\delta \leq K_{\B(0,1/r_2)}(z,w) \leq  K_{\Omega_{t_k}}(z,w)
		\] for some $r_2\in (r_0,1)$,
and we are done.

\end{proof}
	
	\begin{proposition}\label{Prop:persistence-of-geodesic}
		Let $\mathcal F$ be the family of all geodesic rays $\gamma:[0,+\infty)\to \Omega_0$ such that $\lim_{t\to +\infty}\gamma(t)=e_1$. Let $\{t_k\}\subset(0,1)$ be a strictly increasing sequence converging to $1$. Then, 
		for every $w_0\in \B^N$ there exists a sequence of  geodesics $\{\gamma_k\}\subset\mathcal F$ such that $\{A_{t_k}^{-1}(\gamma_k)\}$ converges uniformly on compacta to a geodesic ray $\eta:[0,+\infty)\to\B^N$ such that $\eta(0)=w_0$ and $\lim_{t\to+\infty}\eta(t)=e_1$. 
	\end{proposition}
	\begin{proof}
		Fix $w_0\in \B^N$. Since $A_{t_k}^{-1}$ is an automorphism of $\B^N$, it follows that there exists a sequence $\{z_k\}_{k>k_0}$, converging to $e_1$ such that $A^{-1}_{t_k}(z_k)=w_0$ for all $k>k_0$. Let $\gamma_k\in\mathcal F$ be such that $\gamma_k:[0,+\infty)\to \Omega_0$ satisfies $\gamma_k(0)=z_k$ and $\lim_{t\to+\infty}\gamma_k(t)=e_1$. 
		
		Then $\eta_k:=A^{-1}_{t_k}\circ\gamma_k:[0,+\infty)\to \Omega_{t_k}$ is a geodesic in $\Omega_k$ such that $\eta_k(0)=w_0$ and $\lim_{t\to+\infty}\eta_k(t)=e_1$. 
		
Take a sequence $\rho_k>1,$ decreasing to $1$ such that $\Omega_{t_k}\subset \B(0,\rho_k)$. Now, by Lemma~\ref{Lem:metric-conv}, for every $T>0$ and for all $k>k_0$ we have
	\[
		 K_{\B(0,\rho_{k})}(\eta_k(T), w_0)\leq K_{\Omega_{t_k}}(\eta_k(T), w_0)=T.
		\]
		Therefore, $\{\eta_k\}$ is uniformly bounded on compacta. 
Also, for every $s,t\in [0,+\infty)$,
		\[
		K_{\B(0,\rho_{k})}(\eta_k(t), \eta_k(s))\leq K_{\Omega_{t_k}}(\eta_k(t), \eta_k(s))=|t-s|,
		\]
		therefore,  $\{\eta_k\}$  is equicontinuous on compacta. By Arzel\`a-Ascoli theorem, we can thus extract a subsequence uniformly converging on compacta to a continuous curve $\eta:[0,+\infty)\to \B^N$ such that $\eta(0)=w_0$. Since $K_{\Omega_{t_k}}(\cdot,\cdot)$ converges uniformly on compacta to $K_{\B^N}(\cdot,\cdot)$ by Lemma~\ref{Lem:metric-conv}, it follows that $\eta$ is a geodesic in $\B^N$. Moreover, by (1) of Lemma~\ref{Lem:scaling}, it follows that $\lim_{t\to+\infty}\eta(t)=e_1$. 
		
		Repeating the previous argument for every subsequence of $\{\eta_k\}$, by the uniqueness of geodesics in $\B^N$, we have that $\{\eta_k\}$ converges uniformly on compacta to $\eta$, and we are done.
	\end{proof}
	
	We also need a sort of converse to the previous proposition. For this result we need the following ``visibility lemma'' slightly extending  \cite[Lemma~10]{Kos-Nik-Tho 2023}:
	
	\begin{lemma}\label{lem:vis}
		Let $p,q\in \partial \B^N$, $p\neq q$.  Then for every open neighbourhood $U$ and $V$ of, respectively, $p$ and $q$, such that $U\cap V=\emptyset$, there exists a compact set $L\subset\B^N$ such that for every $t\in [0,1)$ and every  geodesic $\gamma$ of  $\Omega_{t}$ which joins a point in $U$ with a point in $V$, the image of $\gamma$ intersects $L$.
	\end{lemma}
	\begin{proof}
		In case $p,q\neq -e_1$, the result follows immediately from \cite[Lemma~5]{Kos-Nik-Tho 2023}. 
		
		In case $p=-e_1$, we argue by contradiction. If there exists a sequence $\{t_k\}$ converging to $1$ and geodesics $\gamma_k$ of $\Omega_{t_k}$ joining a point in $U$ to a point of $V$ such that the images of $\gamma_k$ escape all compact subsets of $\B^N$, then we can find a sequence $s_k>0$ such that $\{\gamma_k(s_k)\}$ converges to some point $q'\in \partial\B^N\setminus\{-e_1, q\}$. But then, we can apply \cite[Lemma~10]{Kos-Nik-Tho 2023} to $q, q'$ and reach a contradiction.
	\end{proof} 
	
	As a matter of notation, we let $\pi_1:\C^N\to \C$ be the projection on the first component.
	
	\begin{proposition}\label{Prop:convergence-of-geodesic}
		Let $\mathcal F$ be the family of all geodesic rays $\gamma:[0,+\infty)\to \Omega_0$ such that $\lim_{t\to +\infty}\gamma(t)=e_1$. Let $\{t_k\}\subset(0,1)$ be a strictly increasing sequence converging to $1$ and let $\{\gamma_k\}\subset\mathcal F$ be such that there is $0<\epsilon<1$ so that for every $k$ 
		\[
		0<\Re \pi_1(A_{t_k}^{-1}(\gamma_k(0)))<1-\epsilon,
		\] 
		and $\{A_{t_k}^{-1}\gamma_k(0)\}$ is compactly divergent in $\B^N$.
		Then, there exists a subsequence of $\{A_{t_k}^{-1}\circ \gamma_k\}$ such that, after a suitable affine reparametrization, converges uniformly on compacta to a geodesic line  $\eta$  in $\B^N$ joining $e_1$ to a point $\xi\in\partial\B^N$ with $0\leq \Re \pi_1(\xi)\leq 1-\epsilon$.
	\end{proposition}
	\begin{proof} Since $\{\pi_1\circ A_{t_k}^{-1}\}$ converges uniformly on compacta of $\overline{\B^N}\setminus\{e_1\}$ to the constant map $\zeta\mapsto -1$, it follows that 
		$\lim_{t\to+\infty} \gamma_k(0)=e_1$. 
		
		Up to subsequences, we can assume that $\{A_{t_k}^{-1}(\gamma_k(0))\}$ converges to a point $\xi\in\partial\B^N$ such that $0\leq \Re \pi_1(\xi)\leq 1-\epsilon$, in particular, $\xi\neq e_1$. By Lemma~\ref{lem:vis},   there exists a compact subset $K\subset \B^N$ such that $A_{t_k}^{-1}(\gamma_k([0,+\infty))\cap K\neq\emptyset$ for all $k$ sufficiently large. 
		
		Let $s_k>0$ be such that $A_{t_k}^{-1}(\gamma_k(s_k))\in K$ for $k$ sufficiently large. Let $\sigma_k:[-s_k,+\infty)\to \Omega_{t_k}$ be the geodesic ray defined by $\sigma_k(t):=A_{t_k}^{-1}(\gamma_k(t-s_k))$. Note that $\{\sigma_k(-s_k)\}$ converges to $\xi$ and that $\sigma_k(0)\in K$ for $k$ sufficiently large. Since for any $\rho>1$
		\[
		K_{\B(\rho)}(\sigma_k(0),\sigma_k(-s_k))\leq K_{\Omega_{t_k}}(\sigma_k(0),\sigma_k(-s_k))=s_k,
		\] providing that $k$ is big enough,
		it follows  that $\lim_{k\to\infty}s_k=+\infty$.
		
		Arguing as in the proof of Proposition~\ref{Prop:persistence-of-geodesic} it is easy to see that, up to subsequences, $\{\sigma_k\}$ converges uniformly on compacta to a geodesic line $\eta:(-\infty,+\infty)\to\B^N$ such that $\lim_{t\to+\infty}\eta(t)=e_1$. 
		
		We are left to show that $\lim_{t\to-\infty}\eta(t)=\xi$. Suppose this is not the case and that $\lim_{t\to-\infty}\eta(t)=q\in\partial\B^N\setminus\{\xi\}$. Let $V$ and $U$ be open neighborhoods of $\xi$ and $q$ respectively such that $V\cap U=\emptyset$. 
		
		Fix  $M\in \N$, such that $\eta(-M)\in U$. Since $\{\sigma_k\}$ converges uniformly on compacta to $\eta$, there exists $k(M)$ so that, for $k\geq k(M)$ we have  $-s_k<-M$, $\sigma_k(-M)\in U$ and $\sigma_k(-s_k)\in V$. 
		
		Therefore, for all $k\geq k(M)$, $\sigma_k:[-s_k, -M]\to \Omega_{t_k}$ is a geodesic  which joins a point of $U$ to a point of $V$. By Lemma~\ref{lem:vis},   there exists a compact subset $L\subset \B^N$ and $-s_{M,k}\in (-s_k, -M)$ such that $\sigma_k(-s_{M,k})\in L$ for all $k\geq k(M)$. Hence,
		\[
		M\leq K_{\Omega_{t_k}}(\sigma_k(0),\sigma_k(-M))+K_{\Omega_{t_k}}(\sigma_k(-M),\sigma_k(-s_{M,k}))=K_{\Omega_{t_k}}(\sigma_k(-s_{M,k}), \sigma_k(0))\leq C,
		\]
		for some constant $C>0$ independent of $M$ and $k$, which gives a contradiction for $M\to \infty$.
	\end{proof}
	
	A similar argument, that we leave to the reader (see also \cite[Lemma~6]{Kos-Nik-Tho 2023}) gives:
	\begin{proposition}\label{Prop:convergence-of-single-geodesic}
		Let $\mathcal F$ be the family of all geodesic rays $\gamma:[0,+\infty)\to \Omega_0$ such that $\lim_{t\to +\infty}\gamma(t)=e_1$. Let $\{t_k\}\subset(0,1)$ be a strictly increasing sequence converging to $1$ and let $\{\gamma_k\}\subset\mathcal F$ such that $\{\gamma_k(0)\}$ is relatively compact in $\B^N$.
		Then, there exist a subsequence $\{k_m\}$ and $s_m\geq 0$ with $\lim_{m\to\infty}s_m=+\infty$ such that   $\{A_{t_{k_m}}^{-1}(\gamma_{k_m}(t+s_m))\}$ converges uniformly on compacta to a geodesic line  $\eta:\R\to \B^N$ whose image is $(-1,1)e_1$. 
		
		Conversely, if  $s_k>0$ is chosen so that $\{A_{t_k}^{-1}(\gamma_k(s_k))\}$ is relatively compact in $\B^N$, then, up to subsequence, $\{A_{t_k}^{-1}(\gamma_k(s_k+t))\}$ converges  uniformly on compacta to a geodesic line  $\eta:\R\to \B^N$ whose image is $(-1,1)e_1$. 
	\end{proposition}

	\section{The proof of Theorem~\ref{main-intro-inner} and Theorem~\ref{main-intro}}\label{section:main}
	
	\begin{proof}[Proof of Theorem~\ref{main-intro-inner}]
		We first show that $F$ is proper.
		
		For each $\gamma\in \mathcal F$ we can assume that $\gamma(0)=p$. We need to show that if $\{z_k\}$ is a sequence in $D$ converging to some $q\in \partial D$ then $\{F(z_k)\}$ accumulates to $\partial G$.
		
		To this aim, let $\gamma_k\in \mathcal F$ be such that $\gamma_k(T_k)=z_k$ for some $T_k\geq 0$. Note that $\lim_{k\to +\infty}T_k=+\infty$. Indeed, 
		\[
		T_k=K_D(\gamma_k(0), \gamma_k(T_k))=K_D(p, z_k),
		\]
		and $\lim_{k\to \infty} K_D(p, z_k)=+\infty$ since $D$ is complete hyperbolic. By hypothesis,
		\[
		K_G(F(p), F(z_k))=K_G(F(\gamma_k(0)), F(\gamma_k(T_k)))=K_D(\gamma_k(0), \gamma_k(T_k))=T_k,
		\]
		hence $\lim_{k\to \infty}K_G(F(p), F(z_k))=+\infty$, which, since $G$ is a domain (hence connected), implies that $\{F(z_k)\}$ accumulates on $\partial G$.
		
		Now we show that $F$ is actually a biholomorphism onto $G$. 
		
		If there were $q\in D$, $q\neq p$ such that $F(p)=F(q)$ then a geodesic $\gamma\in\mathcal F$ joining $p=\gamma(0)$ with $q=\gamma(t)$, for some $t>0$, could not be mapped onto a geodesic as $F\circ\gamma$ were not injective since $F(\gamma(0))=F(\gamma(t))$. Thus $F^{-1}(F(p))=\{p\}$. 
		
		We claim that the differential $dF_p$ is non degenerate. Indeed, if $dF_p$ is degenerate  there is a vector $0\neq X\in\mathbb C^n$ with $dF_p(X)=0$. Expanding, we have
		\begin{equation}\label{equation:Taylor}
		F(p+tX)=F(p)+O(t^2).
		\end{equation}
		But our assumption on preserving  geodesics of $\mathcal F$  implies that \begin{equation}\label{eq:equal-ista-DG}
		K_D(p,p+tX)=K_G(F(p),F(p+tX)),
		\end{equation}
		for $t>0$ small so that $p+tX\in D$. 
		
		Now, since $D$ is complete hyperbolic and then the Euclidean topology of $D$ is equivalent to the topology induced by $K_D$, there exists $c>0$ such that $K_D(p,p+tX)\geq  c\|p-(p+tX)\|$. Also, by definition of Kobayashi pseudo-distance, there exists $C>0$ such that $K_G(F(p),w)\leq C\|F(p)-w\|$ for $w$ close to $F(p)$. Therefore, dividing  \eqref{eq:equal-ista-DG} by $t>0$, and using \eqref{equation:Taylor}, we have
		\[
		c\|X\|\leq \frac{C}{t} \|F(p)-F(p+tX)\|=O(t),
		\]
		which, for $t\to 0^+$ gives a contradiction, and hence $dF_p$ is non degenerate. 
		
		Thus, the argument principle (see, {\sl e.g.}, \cite[Prop.~2 p.48]{Da}) implies that, for every open set $W\subset\subset D$ such that $\overline{W}$ is diffeomorphic to a ball and such that $\partial W$ has the same orientation of the $2N-1$-dimensional sphere $\mathbb S^{2N-1}$, the map $\frac{F(z)-F(p)}{\|F(z)-F(p)\|}:\partial W\to \mathbb S^{2N-1}$  has degree $1$.
		
		Now, let $w_0\in G$. Let $\eta:[0,1]\to G$ be a continuous injective curve such that $\eta(0)=F(p)$ and $\eta(1)=w_0$. Since $F$ is proper, $F^{-1}(\eta([0,1]))$ is contained in a compact set $K\subset\subset D$. Let $W$ be an open set as before such that $K\subset W$. By construction, $\|F(z)-\eta(t)\|>0$ for all $z\in\partial W$ and $t\in [0,1]$. Hence, $S_t:=\frac{F(z)-\eta(t)}{\|F(z)-\eta(t)\|}:\partial W \to \mathbb S^{2N-1}$  is well defined for $t\in [0,1]$ and homotopic to $S_0$. Therefore, the degree of $S_t$ is $1$ for all $t\in [0,1]$. Hence, again by the argument principle, $F^{-1}(w_0)$ contains exactly only one point. Thus, $F$ is a biholomorphism onto $G$.
	\end{proof}
	
	In order to  prove Theorem~\ref{main-intro} we need a lemma:
	
	\begin{lemma}\label{Lem:ball}
		Let $\Phi: \B^N\to \B^N$. Let $\mathcal G$ be the family of all geodesic rays landing at $e_1$. If $\Phi\circ \gamma$ is a geodesic of $\B^N$ for all $\gamma\in\mathcal G$ then $\Phi$ is an automorphism of $\B^N$.
	\end{lemma}
	\begin{proof}
		Let $\varphi:\D\to\B^N$ be a complex geodesic containing $e_1$ in its closure---that is, $\varphi$ is an injective holomorphic map such that its image is the intersection of a complex affine line with $\B^N$ and the closure of its image contains $e_1$. Hence, $\varphi:\D\to \B^N$ is an affine map. It is well known that all geodesics rays of $\B^N$ landing at $e_1$ are contained in one such a complex geodesic. We are going to show that $\Phi\circ \varphi$ is a complex geodesic of $\B^N$ for every $\varphi\in\mathcal G$. Then the result follows from \cite[Theorem~1.2]{BKZ}.
		
	To this aim, let $\varphi$ be a complex geodesic in $\B^N$ whose closure contains $e_1$. Up to post-composing $\varphi$ with an automorphism of $\B^N$ fixing $e_1$, we can assume that $\varphi(\zeta)=\zeta e_1$. Since $[0,1)\ni r\mapsto re_1$ is a geodesic of $\B^N$, it follows that $[0,1)\ni r\mapsto F(re_1)$ is a geodesic in $\B^N$. Up to composing with an automorphism of $\B^N$, we can assume that $F([0,1)e_1)=[0,1)e_1$. Let $F_1$ be the first component of $F$. Thus $F_1:\D\to \D$ is a holomorphic map which maps a geodesic into a geodesic of $\D$. By the Schwarz-Pick lemma, $F_1(\zeta e_1)=\zeta$ for all $\zeta\in \D$. Since $|F_1(\zeta e_1)|^2+|F'(\zeta e_1)|^2<1$ for all $\zeta\in \D$ (here $F'=(F_2,\ldots, F_N)$), we have that $\lim_{\D\ni \zeta\to e^{i\theta}}|F'(\zeta e_1)|=0$ for all $\theta\in \R$. By the maximum principle, $F'(\zeta e_1)\equiv 0$. Therefore, $F(\zeta e_1)=\zeta e_1$ for all $\zeta\in \D$ and we are done.
	\end{proof}

	\begin{proof}[Proof of Theorem~\ref{main-intro}]
		We divide the proof into some Steps:
		\smallskip
		
		{\sl Step 1.} {\sl There exists $p'\in\partial{ G}$ such that $\{F(\gamma(t))\}$ converges  to $p'$ for all $\gamma\in \mathcal F$}.
		
		\smallskip
		
If $\gamma\in\mathcal F$, by Proposition~\ref{Prop:geo-in-sp}.(3) there exists $p'\in\partial G$ such that $\gamma(t)$ converges to $p'$ as $t\to +\infty$. If $\eta\in \mathcal F$ then by Proposition~\ref{Prop:geo-in-sp}.(4) there exists $C>0$ such that $\lim_{t\to+\infty}K_D(\gamma(t), \eta(t))\leq C$. Hence, $\lim_{t\to+\infty}K_G(F(\gamma(t)), F(\eta(t)))\leq C$, which, by the standard estimates on the Kobayashi distance in strongly pseudoconvex domains (see also  D'Addezio's Lemma  \cite[Lemma~A.2]{BG}), implies that $F(\eta(t))$ converges to  $p'$ as $t\to+\infty$.

		\smallskip
		
		{\sl Step 2. $F$ is proper (and hence surjective).} 
		
		\smallskip

		Arguing as in the proof of Theorem~\ref{main-intro-inner}, we see that if $\{z_k\}\subset D$ converges to some $\xi\in\partial D\setminus\{p\}$, then $\{F(z_k)\}$ accumulates at $\partial G$. 
		
		In case $\{z_k\}\subset D$ converges to $p$, let $\gamma_0:[0,+\infty)\to D$ be a geodesic such that $\lim_{t\to+\infty}\gamma(t)=p$. 
		
		If there exists $C>0$ such that for every $k$ there exists $s_k>0$ with $K_D(\gamma(s_k), z_k)<C$ (in particular, $\{s_k\}$ converges to $+\infty$), then 
		\[
		K_G(F(\gamma(s_k)), F(z_k))\leq K_D(\gamma(s_k), z_k)<C,
		\]
		and, as before, it follows that $\{F(z_k)\}$ converges to $p'$. 
		
		Therefore, we can assume that 
		\begin{equation}\label{Eq:seq-tg}
		\lim_{k\to +\infty}K_D(\gamma_0([0,+\infty)), z_k)=+\infty.
		\end{equation}
		We argue by contradiction and assume that $\{F(z_k)\}$ is relatively compact in $G$.

		Let $T_D:D\to D_0$ and $T_G:G\to G_0$ be the maps defined in Section~\ref{Sec:scaling}, mapping, respectively, $p$ to $e_1$ and $p'$ to $e_1$, such that $G_0, D_0$ are contained in some balls  and the defining functions of $D_0, G_0$ at $e_1$ are as in \eqref{eq:def}. Let $F_0:=T_G\circ F\circ T^{-1}_D$. Since $T_D$  extends biholomorphically through the boundary of $D$, the family $T(\mathcal F)$ coincides with the family $\mathcal F'$ of geodesic rays $\gamma:[0,+\infty)\to D_0$ such that $\lim_{t\to+\infty}\gamma(t)=e_1$, and $F_0\circ \gamma$ is a geodesic in $G_0$ for all $\gamma\in \mathcal F'$.
		
		We let $\tilde z_k:=T_D(z_k)$ and $\tilde\gamma_0=T_D\circ \gamma_0$. For $t\in (-1,1)$ let $A_{t}$ be as in \eqref{eq:aut}. 
		
		Since $\{\tilde z_k\}$ converges to $e_1$, we can find a strictly increasing sequence $\{t_k\}$ converging to $1$ such that for all $k$,
		\begin{equation}\label{Eq:how-fast-rescaling}
		0<\Re \pi_1 (A_{t_k}^{-1}(\tilde z_k))<\frac{1}{2},
		\end{equation}
		where $\pi_1$ is the projection on the first component.
		
		Let $D_k:=A^{-1}_{t_k}(D_0)$. Also, let  $\tau_k:=A_{t_k}^{-1}\circ \tilde\gamma_0$. Finally, if $\eta_k\in\mathcal F'$ is such that $\eta_k(0)=\tilde z_k$, we let $\sigma_k:=A_{t_k}^{-1}\circ \eta_k$.
		
		Note that $\tau_k$ and $\sigma_k$ are geodesic rays in $D_k$ landing at $e_1$.
		
		By Proposition~\ref{Prop:convergence-of-single-geodesic},  up to  subsequences and a reparametrization $t\mapsto t+s_k$, for some $s_k>0$ such that $\lim_{k\to\infty}s_k=+\infty$, $\{\tau_k(t+s_k)\}$ converges to a geodesic line $\tau:(-\infty,+\infty)\to \B^N$ whose image is $(-1,1)e_1$. 
		
		Now, we claim that the sequence formed by $\sigma_k(0)=A_{t_k}^{-1}(\eta_k(0))=A_{t_k}^{-1}(T_D(z_k))$ is compactly divergent in $\B^N$. Indeed, assume that there is a subsequence,  still denoted by $\{\sigma_k(0)\}$, which converges to some $\xi_0\in \B^N$. Hence, there is $C>0$ such that $K_{\B^N}(\tau(0), \xi_0)<C$. Therefore, for $k$ sufficiently large, 
		\[
		K_{D_k}(\tau_k(s_k), \sigma_k(0))\leq 2C.
		\]
		Hence,
		\[
		K_D(\gamma_0(s_k), z_k)=K_{A^{-1}_{t_k}(T_D(D))}(A_{t_k}^{-1}(T_D(\gamma_0(s_k))), A_{t_k}^{-1}(T_D(z_k)))=K_{D_k}(\tau_k(s_k), \sigma_k(0))\leq 2C,
		\]
		contradicting \eqref{Eq:seq-tg}. 
		
		Thus, $\{\sigma_k(0)\}$ is compactly divergent in $\B^N$ and by \eqref{Eq:how-fast-rescaling} we can apply Proposition~\ref{Prop:convergence-of-geodesic}, and, up to subsequences and reparametrization, we see that $\{\sigma_k\}$ converges to a geodesic line  $\sigma:(-\infty,+\infty)\to\B^N$ such that $\lim_{t\to+\infty}\sigma(t)=e_1$ and $\lim_{t\to -\infty}\sigma(t)=q$ for some $q\neq -e_1$.
		
		In particular, by the uniqueness of geodesic lines in $\B^N$ landing at $e_1$, we have
		\begin{equation}\label{Eq:diff-geo}
		\tau((-\infty,+\infty))\cap \sigma((-\infty,+\infty))=\emptyset.
		\end{equation}

		By \eqref{Eq:seq-tg},
		\[
		\lim_{k\to +\infty}K_{D_k}(\tau_k([0,+\infty)), \sigma_k(0))=+\infty.
		\]
		
		Now, since $\{s_k\}$ converges to $+\infty$, it follows that $F_0(\tilde\gamma_0(s_k))$ converges to $e_1$. Hence, for every $k$ we can find $a_k\in [0,1)$ such that $\{A_{a_k}^{-1}(F_0(\tilde\gamma_0(s_k)))\}$ is  contained in a compact set $L\subset\B^N$.
		
		Let $G_k:=A^{-1}_{a_k}(G_0)$ and
		\[
		F_k:=A^{-1}_{a_k}\circ F_0\circ A_{t_k}:D_k\to G_k.
		\]
		By construction, $\{\tau_k(s_k)\}$ converges to $\tau(0)$ and $F_k(\tau_k(s_k))=A_{a_k}^{-1}(F_0(\tilde\gamma_0(s_k))\in L$.
		
		Hence, again up to subsequences, we can assume that $\{F_k\}$ converges uniformly on compacta to a holomorphic self-map $\Phi:\B^N\to \B^N$ with the property that $\Phi$ maps  $\tau(0)$ in $L$. 
		
		We claim that, for every geodesic ray $\eta:[0,+\infty)\to\B^N$ such that $\lim_{t\to+\infty}\eta(t)=e_1$ then $\Phi\circ\eta$ is a geodesic  in $\B^N$.
		
		Indeed, fix such a $\eta$ and let $w_0=\eta(0)$. By Proposition~\ref{Prop:persistence-of-geodesic}, there exists a sequence $\{\gamma_k\}\subset \mathcal F$ such that $\{A_{t_k}^{-1}\circ \gamma_k\}$ converges uniformly on compacta to a geodesic ray starting at $w_0$ and landing at $e_1$. By the uniqueness of geodesics in $\B^N$, such a geodesic ray is $\eta$ itself. 
		
		Therefore, by Lemma~\ref{Lem:metric-conv}, for every fixed $t,s\in [0,+\infty)$, we have
		\begin{equation*}
		\begin{split}
		K_{\B^N}(\Phi(\eta(t)),\Phi(\eta(s)))&=\lim_{k\to\infty}K_{G_k}(F_k(A_{t_k}^{-1}(\gamma_k(t))), F_k(A_{t_k}^{-1}(\gamma_k(s))))\\
		& =\lim_{k\to\infty}K_{A^{-1}_{a_k}(G_0)}(A^{-1}_{a_k}(F_0(\gamma_k(t))), A^{-1}_{a_k}(F_0(\gamma_k(s))))\\
		&=\lim_{k\to\infty}K_{G_0}(F_0(\gamma_k(t)), F_0(\gamma_k(s)))=\lim_{k\to\infty}|t-s|=|t-s|.
		\end{split}
		\end{equation*}
		It follows by Lemma~\ref{Lem:ball} that $\Phi$ is an automorphism of $\B^N$. In particular, by \eqref{Eq:diff-geo},
		\begin{equation}\label{Eq:false-geo-diff}
		\Phi(\tau((-\infty,+\infty)))\cap \Phi(\sigma((-\infty,+\infty)))=\emptyset.
		\end{equation}
		
		Now, $\{F_k(\tau_k(t+s_k))\}$ converges uniformly on compacta to $\Phi(\tau((-\infty,+\infty)))$. Since 
		\[
		F_k\circ \tau_k(t+s_k)=(A^{-1}_{a_k}\circ F_0\circ A_{t_k})\circ (A_{t_k}^{-1}\circ \tilde \gamma_0)(t+s_k)=A^{-1}_{a_k} \circ F_0 \circ \tilde \gamma_0(t+s_k),
		\]
		and $t\mapsto F_0 \circ \tilde \gamma_0(t+s_k)$ is a geodesic ray in $G_0$ landing at $e_1$,  
		by  Proposition~\ref{Prop:convergence-of-single-geodesic} we have
		\begin{equation}\label{Eq:image-tau-Phi}
		\Phi(\tau((-\infty,+\infty)))=(-1,1)e_1.
		\end{equation}
		On the other hand, since $\{F(z_k)\}$ is relatively compact in $G$---hence $\{F_0(\tilde z_k)\}$ is relatively compact in $G_0$, it follows as before and by  Proposition~\ref{Prop:convergence-of-single-geodesic} that
		\begin{equation}\label{Eq:image-sigma-Phi}
		\Phi(\sigma((-\infty,+\infty)))=(-1,1)e_1.
		\end{equation}
		But \eqref{Eq:image-tau-Phi} and \eqref{Eq:image-sigma-Phi} contradict \eqref{Eq:false-geo-diff}. Therefore, $F$ is proper. In particular, $\det(dF_z)\not\equiv 0$.
		
		\smallskip

		{\sl Step 3. $F$ is injective.} 
		
		\smallskip

		As in Step. 2, we let $D_0:=T_D(D)$, $G_0:=T_G(G)$ and $F_0:=T_G  \circ F \circ T_D^{-1}:D_0\to G_0$.  The statement is thus equivalent to show that $F_0$ is injective. 
		
		By Step. 3, $\mathcal Z:=\{z\in D_0: \det(dF_0)_z=0\}$ is a (possibly empty) analytic subvariety of $D_0$---if $D_0$ and $G_0$ are sufficiently smooth (see, {\sl e.g.}, \cite{Fo}) then $\mathcal Z=\emptyset$ and $F_0$ is a covering map, but we do not need this in our argument. Since $F_0$ is proper by Step.~2, $F_0(\mathcal Z)$ is an analytic subvariety of $G_0$ by Remmert's Theorem. In particular, $G_0':=G_0\setminus F_0(\mathcal Z)$ is a connected open and dense set in $G_0$ and $F:D_0\setminus \mathcal Z\to G_0'$ is a finite covering map. 
		
		For $w\in G_0$, let $\nu(w)$ be the number of points of $F^{-1}(w)$. Thus, $\nu(w)=\nu(w')=:\nu$ for all $w, w' \in G'_0$,  and $1\leq \nu(w)\leq \nu$ for all $w\in F_0(\mathcal Z)$. Hence, it is enough to show that $\nu=1$. 

Arguing by contradiction, we assume that $\nu\geq 2$. Fix $\tilde w\in G'_0$ such that the vector $\tilde w-e_1$ is almost complex tangential. Then by \cite[Corollary~4]{Kos-Nik-Tho 2023}, for $w$ sufficiently close to $\tilde w$ there exists a unique geodesic ray $\eta_w:[0,+\infty)\to G_0$ such that $\eta_w(0)=w$ and $\lim_{t\to+\infty} \eta_w(t)=e_1$. Moreover, if two such $\eta_w$ and $\eta_{w'}$ intersect at one point, then the range of one is contained in the other. Also $\eta_w(t)$ varies smoothly with respect to $w$ and $t$. Take an open subset $\Delta$ in a real hyperplane such that any $w\in \Delta$ gives rise to exactly one $\eta_w$. Then, $\Delta \times [0,\infty)\to D$, $(w,t)\mapsto \eta_w(t)$ is a $\mathcal C^1$ smooth diffeomorphism onto its image, consequently it is locally bi-Lipschitz. The set $F(\mathcal Z)$ has real codimension at least $2$, while the real codimension of $\Delta$ is clearly 1. Thus, the set of $w\in \Delta$ such that the range of $\eta_w$ does not intersect $F(\mathcal Z)$ has real Hausdorff codimension at least 1. In particular, it is non-empty. Fix one such  a $w_0$ and let $\eta:=\eta_{w_0}$. By  construction, $\eta([0,+\infty))\subset G_0'$.

		Since we are assuming that $\nu\geq 2$, there exists $z_1, z_2\in D_0$, $z_1\neq z_2$ such that $F(z_1)=F(z_2)=w_0$. 
		
		Let $\gamma_j:[0,+\infty)\to D_0$ be geodesic rays such that $\gamma_j(0)=z_j$ and $\lim_{t\to+\infty}\gamma_j(t)=e_1$, $j=1,2$. By hypothesis, $F_0\circ \gamma_j$ is a geodesic in $G_0$, and, by Step.~1, $\lim_{t\to+\infty}F_0(\gamma_j(t))=e_1$ for $j=1,2$. Moreover, by construction, $F_0(\gamma_1(0))=F_0(\gamma_2(0))=w_0$. Therefore, by the uniqueness of $\eta$, 		\[
		F_0(\gamma_j(t))=\eta(t), \quad \hbox{for $t\geq 0$ and $j=1,2$}. 
		\]
		
		It follows from  $\eta([0,+\infty))\subset G_0'$ that there are no $s,t>0$ such that $z=\gamma_1(s)=\gamma_2(t)$. Take any sequence $\{s_k\}$ converging to $+\infty$. Then $\gamma_1(s_k)\neq \gamma_2(s_k)$.
		
		Let $\{a_k\}$ be a strictly increasing sequence of positive numbers  converging to $1$ such that $\{A_{a_k}^{-1}(\gamma_1(s_k))\}$ is relatively compact in $\B^N$, where $A_{a_k}$ is given by \eqref{eq:aut}. 
		
		Likewise, we can find a sequence $\{b_k\}$, $0<b_k<1$ converging to $1$ such that $\{A_{b_k}^{-1}(\eta(s_k))\}$ is relatively compact in $\B^N$.
		
		Let $D_k:=A_{a_k}^{-1}(D_0)$, $G_k:=A_{b_k}^{-1}(G_0)$ and $F_k:=A^{-1}_{b_k}\circ F_0\circ A_{a_k}:D_k\to G_k$.
		
		By Proposition~\ref{Prop:convergence-of-single-geodesic}, up to subsequences and reparametrization, we can assume that $\{A_{a_k}^{-1}\circ \gamma_1\}$ converges uniformly on compacta to a geodesic $\sigma_1$ of $\B^N$ whose image is $(-1,1)e_1$.
		
		By Proposition~\ref{Prop:geo-in-sp}.(4), there exists $C>0$ such that for every $s_k$ we can find $u_k>0$ such that $K_{D_0}(\gamma_1(s_k), \gamma_2(u_k))<C$. It follows, again by Proposition~\ref{Prop:convergence-of-single-geodesic}, that, up to subsequences and reparametrization, also $\{A_{a_k}^{-1}\circ \gamma_2\}$ converges uniformly on compacta to a geodesic $\sigma_2$ of $\B^N$  whose image is $(-1,1)e_1$. In particular,
		\[
		\sigma_1((-\infty,+\infty))=\sigma_2((-\infty,+\infty)).
		\]
		
		A similar argument shows that, up to subsequences and reparametrization,  $\{A_{b_k}^{-1}\circ \eta\}$ converges uniformly on compacta to a geodesic $\tau$ of $\B^N$  whose image is $(-1,1)e_1$. 
		
		Arguing as in Step.~2, and up to subsequences, we can also assume that  $\{F_k\}$ converges uniformly on compacta of $\B^N$ to an automorphism $\Phi$ of $\B^N$.
		
		Since $F_k(A_{a_k}^{-1}(\gamma_j((0+\infty))))=A_{b_k}^{-1}(\eta(0,+\infty))$ for $j=1,2$ and for all $k$, it follows that $\Phi(\sigma_j((-\infty,+\infty)))=\tau((-\infty,+\infty))$, $j=1,2$. 
		
		Let $\zeta_k^1:=A_{a_k}^{-1}(\gamma_1(s_k))$ and let $x_0\in (-1,1)e_1$ be the limit of $\{\zeta_k^1\}$. Let $\zeta_k^2\in A_{a_k}^{-1}(\gamma_2([0,+\infty)))$ be such that 
		$F_k(\zeta_k^2)=F_k(\zeta_k^1)$ for all $k$ (it exists since $F_0(\gamma_1([0,+\infty))=F_0(\gamma_2([0,+\infty))$). 
		
		We claim that $\{\zeta_k^2\}$ converges to $x_0$. To this aim, it is enough to show that $\{\zeta_k^2\}$ is relatively compact in $\B^N$, because, if so and, up to subsequences, $x_1\in\B^N$ is the limit of $\{\zeta_k^2\}$ then 
		\[
		\Phi(x_1)=\lim_{k\to +\infty}F_k(\zeta^2_k)=\lim_{k\to +\infty}F_k(\zeta^1_k)=\Phi(x_0), 
		\]
		and, since $\Phi$ is an automorphism, $x_0=x_1$.
		
		In order to show that $\{\zeta_k^2\}$ is relatively compact in $\B^N$, since $\{A_{a_k}^{-1}\circ \gamma_2\}$ converges, up to reparametrization, to $\sigma_2$ whose image is equal to that of $\sigma_1$, there exists a sequence $\{\xi_k\}$  converging to $x_0$ and such that $\xi_k\in A_{a_k}^{-1}(\gamma_2([0,+\infty))$ for all $k$. Thus,
		\begin{equation*}
		K_{D_k}(\xi_k, \zeta_k^2)=K_{G_k}(F_k(\xi_k), F_k(\zeta_k^2))
		=K_{G_k}(F_k(\xi_k)), F_k(\zeta_k^1)).
		\end{equation*}
		Since $\{\xi_k\}$ and $\{\zeta_k^1\}$ converge to $x_0$, it follows by Lemma~\ref{Lem:metric-conv} that the last quantity is bounded. Hence, $\{\xi_k\}$ is relatively compact in $\B^N$ and hence it converges to $x_0$.
		
		Since $\gamma_1(s_k)\neq \gamma_2(s_k)$, we see that $\zeta_k^1\neq \zeta_k^2$ for all $k$. 
		
		Now, $\{F_k\}$ converges uniformly on compact to $\Phi$, which is injective. Hence, for $\epsilon>0$ so that $B(x_0,\epsilon):=\{z\in \C^N: \|z-x_0\|<\epsilon\}\subset\B^N$, there exists $k_0$ such that $F_k$ is injective on $B(x_0,\epsilon)$ for $k\geq k_0$. However, for $k$ large, $\zeta_k^1, \zeta_k^2\in B(x_0,\epsilon)$, $\zeta_k^1\neq \zeta_k^2$ and $F_k(\zeta_k^1)=F_k(\zeta_k^2)$, a contradiction.		
	\end{proof}
	
	\section{Special cases: holomorphic coverings and Reinhardt domains}\label{section:special-cases}
	
	It is natural to explore the possibility of relaxing assumptions on the mapping $F$ (and the domains $D$ and $G$) in results like those in Theorem~\ref{main-intro-inner} and Theorem~\ref{main-intro}, while preserving the same conclusion. However, as we will demonstrate, such assumptions cannot be omitted in general. Below, we provide examples of non-trivial proper holomorphic mappings, which are also holomorphic coverings, between Reinhardt domains (that can be chosen to be strongly pseudoconvex). These mappings do not constitute biholomorphisms, despite being Kobayashi isometries along a complete family of geodesic lines.
		
	First, we present a general result on the behaviour of geodesics under holomorphic coverings.
	
	\begin{proposition}\label{proposition:lifting}
		Lef $\pi:D\to G$ be a holomorphic covering between complete hyperbolic domains in $\mathbb C^n$, $w,z\in G$. Then $\gamma$ is a geodesic segment in $G$ joining $w$ and $z$ if and only if there exist points $\tilde w, \tilde z\in D$ with $\pi(\tilde w)=w$,  $\pi(\tilde z)=z$, and a geodesic segment $\tilde\gamma$ in $D$  such that $\pi\circ \tilde\gamma=\gamma$ and $K_D(w,z)=K_G(\tilde w,\tilde z)$.
		
		In particular, given $w_0\in G$ and $\tilde w_0\in D$ such that $\pi(\tilde w_0)=w_0$,  any geodesic  passing through $w_0\in G$  can be lifted to a geodesic of $D$ passing through $\tilde w_0$.
			\end{proposition}
	\begin{proof} Recall that the holomorphic coverings preserve the Kobayashi metric.
		Denote by $\gamma:[t_0,t_1]\to G$ a geodesic segment joining $w$ and $z$ and denote by $\tilde \gamma$ its lifting to $D$. Let $\tilde \gamma(t_0)=\tilde w$, $\tilde \gamma(t_1)=\tilde z$. Then we have the following inequalities
		\begin{equation*}
		K_G(w,z)\leq K_D(\tilde w,\tilde z)\leq \int_{t_0}^{t_1}k_D(\tilde\gamma(t);\tilde\gamma^{\prime}(t))dt=\int_{t_0}^{t_1}k_G(\gamma(t);\gamma^{\prime}(t))dt,
		\end{equation*}
		which straightforwardly establishes the desired equivalence.			\end{proof}
	
	\begin{remark} Let us consider the following holomorphic covering
		\begin{equation*}
		\{\lambda\in\mathbb C:-\log R<\re(\lambda)<\log R\}=:H_R\owns\lambda\mapsto \exp (\lambda)\in A_R:=\{\lambda \in \mathbb C:1/R<|\lambda|<R\}. 
		\end{equation*}
		Note that the geodesic lines in $H_R$ given by $\R\ni s\mapsto t_0+is$ for 
		$-\log R<t_0<\log R$,   a complete collection of geodesic lines; they are mapped isometrically onto geodesic lines $\{t\exp(is):1/R<t<R\}$, $s\in\mathbb R$ though the mapping itself is not even proper.
	\end{remark}
	
	\begin{remark} More generally, in view of the results above we see that under the assumption that the mapping $F:D\mapsto G$ is a holomorphic covering and $G$ admits a complete collection $\mathcal F$ of geodesic lines then we may lift the collection $\mathcal F$ to a complete family $\widetilde{\mathcal F}$ of geodesic lines so that $F$ maps isometrically elements of $\widetilde{\mathcal F}$ preserving the Kobayashi distance along them. Consequently, when the mapping is not biholomorphic we get a kind of counterexample to the desired biholomorphcity of the mapping $F$. We shall make it very concrete in the next subsection.
	\end{remark}
	
	\subsection{Examples in Reinhardt domains}
	In this part we are dealing with Reinhardt pseudoconvex domains $D\subset\mathbb C_*^n$, $n\geq 1$. In such a situation the domain
	\begin{equation*}
	\log D:=\{\log |z|:=(\log|z_1|,\ldots,\log|z_n|):z\in D\}
	\end{equation*}
	is convex. For simplicity we impose a stronger condition on $D$, namely, we assume that the following is satisfied: $\log D$ is bounded, which means
	that $D$ is additionally bounded and 'separated' from the axes.
	
	Recall that the mapping
	\begin{equation*}
	T_{\log D}\owns u\mapsto\exp u:=(\exp(u_1),\ldots,\exp(u_n))\in D,
	\end{equation*}
where $T_{\log D}:=\log D+i\mathbb R^n$ is a {\it tube domain with the basis $\log D$},
	is a holomorphic covering. A formula describing the behaviour of the Kobayashi distance under holomorphic coverings gives the following
	\begin{equation*}
	K_D(w,z)=\inf\{K_{T_{\log D}}(u,v+i2\pi\nu): \nu\in\mathbb Z^n\},
	\end{equation*}
	where $\exp(u)=w$, $\exp(v)=z$. Note that in our situation the infimum is always attained.
	
	This leads us to the next example.
	
	\begin{example}
		The rays $(0,1)\owns t\mapsto t\omega\in \mathbb D_*$, $|\omega|=1$ are geodesic lines such that we have the equality
		$K_{\mathbb D_*}(t^n\omega^n,s^n\omega^n)=K_{\mathbb D_*}(t\omega,s\omega)$, $0<t,s<1$ showing that the mapping $\mathbb D_*\owns\lambda\mapsto \lambda^n\in\mathbb D_*$ maps isometrically a complete collection of geodesic lines and is not a biholomorphism (if $n\geq 2$).
		
		The analogous result may be proven for the annuli $A(0,r,1)$ and $A(0,r^n,1)$ -- we shall give the details below in a more general setting (in arbitrary dimension).
	\end{example} 
	
	One may also easily see that in the case $x,y\in \log D$ we have the following equality
	\begin{equation}
	K_D(\exp(x),\exp(y))=K_{_{T_{\log D}}}(x,y).
	\end{equation}
	In fact to see the last property it is sufficient to see that any competing analytic disc $f:\mathbb D\to T_{\log D}$ for $x$ and $y+i2\pi\nu$ for some $\nu\in\mathbb Z^n$, {\sl i.e.}, the one with $f(0)=x$, $f(t)=y+i2\pi\nu$, $t>0$ gives rise to a competitor $g(\lambda):=(f(\lambda)+\overline{f(\overline{\lambda})})/2$ with $g(0)=x$, $g(t)=y$, from which the formula follows at once.
	
	Adopting the definition from \cite{Zwo 2021} we call two points $x,y\in\partial \log D$ {\it antipodal} if there is a hyperplane $H$ such that $x+H$ and $y+H$ are not identical and they are supporting hyperplanes of $\log D$. Let us make some comments. First note that the segment $(-1,1)$ is a geodesic line (up to a parametrization) in the strip $H:=\{\lambda\in\mathbb C:-1<\re \lambda<1\}$. Consequently, the open segment joining antipodal points $x,y\in\partial\log D$ is a geodesic line in $T_{\log D}$ and thus by the fact that the map defined by the formula (13) in \cite{Zwo 2021} is a complex geodesic we easily conclude that (up to a parametrization) the curve
	\begin{equation*}
	(-1,1)\owns t\mapsto\exp((x+y)/2)\exp(t(x-y)/2)
	\end{equation*}
	is a geodesic line in $D$.
	
	Now we recall the definition of the monomial mapping $\Phi_A$ (see \cite{Zwo 1999}, \cite{Zwo 2000}) where $A$ is $n\times n$ matrix with integer elements. First for $\alpha\in\mathbb C^n$ we define
	\begin{equation}
	z^{\alpha}:=z_1^{\alpha_1}\cdot\ldots\cdot z_n^{\alpha_n}
	\end{equation}
	for all $z$ with $z_j\neq 0$ if $\alpha_j<0$ (this is always satisfied in our situation).  Then we define
	\begin{equation}
	\Phi_A(z):=(\Phi_{A^1}(z),\ldots,\Phi_{A^n}(z)),
	\end{equation}
where $A^j$ denotes the $j$-th row of $A$.
It is known that $\Phi_A$ is a proper holomorphic mapping of $\mathbb C_*^n$ if and only if $\det A\neq 0$; additionally, its multiplicity equals $|\det A|$ (see e.g. Theorem 2.1 in \cite{Zwo 1999}).
	
Below we consider a non-singular $A$. Note that $\log(\Phi_A(D))=A(\log D)$, so $\Phi_A$ when restricted to $D$ is a proper holomorphic mapping onto the pseudoconvex Reinhardt domain $\Phi_A(D)$. It is also a holomorphic covering.
	
Recall that Proposition~\ref{proposition:lifting} lets us to lift any geodesic segment/ray/line in $\Phi_A(D)$ to a geodesic segment/ray/line in $D$ such that $\Phi$,  restricted to it, is a Kobayashi isometry. 
	
To make the next example concrete let $\log D$ be the unit Euclidean ball and $A=2\mathbb I_n$. We then construct a holomorphic map $\exp(T_{\mathbb B_n^{\mathbb R}})\owns z\mapsto (z_1^2,\ldots,z_n^2)\in \exp(T_{2\mathbb B_n^{\mathbb R}})$
that maps isometrically a complete collection of geodesic lines in $D$ that is obtained by coordinate-wise rotations of geodesics lying in $\mathbb B_n^{\mathbb R}$ and is proper, with multiplicity equal to $2^n$.

Note that $A$ maps antipodal points of $\log D$ to antipodal points in $A(\log D)$. Moreover, we get the following equality for geodesic lines $\gamma$ joining the antipodal points (such that the graph is the segment)
	\begin{equation}
	K_D(\gamma(t),\gamma(s))=K_{\Phi_A(D)}(\Phi_A(\gamma(t)),\Phi_A(\gamma(s))),
	\end{equation}
	which easily shows that the above construction may be easily transferred to Reinhardt domains with  $\log D$ satisfying the following property:	For any $u\in \log D$ there are two antipodal points such that $u$ lies on the segment joining $x$ and $y$.

\end{document}